\documentclass{gtart}

\usepackage[inner=20mm, outer=20mm, textheight=245mm]{geometry}

\usepackage{graphicx}
\usepackage{amsthm}
\usepackage{amsmath, amssymb}

\def\n@te#1{\textsf{\boldmath \textbf{$\langle\!\langle$#1$\rangle\!\rangle$}}\leavevmode}
\def\note#1{\textcolor{red}{\n@te {#1}}}

\def\tri{\mathcal{T}}

\DeclareMathOperator{\genus}{g}

\newcommand{\from}{\colon\thinspace} 
\newcommand{\isom}{\cong} 
\newcommand{\homeo}{\cong} 

\theoremstyle{plain}
\newtheorem{Thm}{Theorem}

\newtheorem{Coro}[Thm]{Corollary}

\newtheorem{Que}[Thm]{Question}

\newtheorem{Con}[Thm]{Construction}

\theoremstyle{definition}

\newtheorem{Def}[Thm]{Definition}

\begin{document}
\title{Computing trisections of 4--manifolds} 
\author{Mark Bell, Joel Hass, J. Hyam Rubinstein and Stephan Tillmann}

\begin{abstract}
Algorithms that decompose a manifold into simple pieces reveal the geometric and topological structure of the manifold, showing how complicated structures are constructed from simple building blocks.
This note describes a way to algorithmically construct a trisection, which describes a $4$--dimensional manifold as a union of three $4$--dimensional handlebodies.
The complexity of the $4$--manifold is captured in a collection of curves on a surface, which guide the gluing of the handelbodies.
The algorithm begins with a description of a manifold as a union of pentachora, or $4$--dimensional simplices.
It transforms this description into a trisection. This results in the first explicit complexity bounds for the trisection genus of a $4$--manifold in terms of the number of pentachora ($4$--simplices) in a triangulation.
\end{abstract}



\maketitle


\vfill

\section{Introduction}

A guiding principle in low-dimensional topology is to find practical algorithms to describe topological or geometric structures on manifolds and to compute invariants, as well as to determine explicit complexity bounds for these algorithms.
The steps in an algorithm reveal the structure of a manifold and the complexity bounds relate the relative difficulty of a wide variety of problems.

Gay and Kirby~\cite{GK} introduced the concept of a \emph{trisection} for arbitrary smooth, oriented closed $4$--manifolds.
This paper is a first step towards a computational theory for understanding $4$--manifolds via trisections.
We use singular triangulations to give a description of a $4$--manifold.
These are well established as a data structure for algorithmic $3$--manifold theory \cite{Thompson-thin-1994, Jaco-0-efficient-2003, Frigerio-constructing-2004, Dunfield-spanning-2011, Li-Heegaard-2011, Schleimer-NP-2011, Coward-reidemeister-2014, Hoffman-verified-2016, Burton-courcelle-2017} and have shown promise for analyzing manifolds in higher dimensions \cite{Budney-Cappell-2012, Casali-classifying-2016, RT-even}.
Budney and Burton~\cite{4d-census} have a census of $4$--manifold triangulations with up to six pentachora.
This is a rich source of examples, and further study or extension of this census requires algorithmic tools.

We develop a theory of colorings for $4$--manifold triangulations, starting with a basic notion of \emph{tricoloring} that encodes suitable maps to the $2$--simplex, and enhancing this to \emph{c-tricoloring} with appropriate connectivity properties and \emph{ts-tricoloring} which completely encodes a trisection.

In dimensions $\le 4$, there is a bijective correspondence between isotopy classes of smooth and piecewise linear structures~\cite{C1, C2}.
All manifolds are assumed to be piecewise linear (PL) in this paper unless stated otherwise.
Our definition and results apply to any compact smooth manifold by passing to its unique piecewise linear structure \cite{Whitehead1940}.

\newpage

\begin{Def}[(Trisection of closed manifold)]\label{def:multisection}
Let $M$ be a closed, connected, piecewise linear $4$--manifold.
A \emph{trisection} of $M$ is a collection of three PL submanifolds $H_0, H_1, H_2 \subset M$, subject to the following four conditions:
\begin{enumerate}
\setlength\itemsep{0em}
\item  Each $H_i$  is PL homeomorphic to a standard PL $4$--dimensional $1$--handlebody of genus $g_i$.
\item The handlebodies $H_i$ have pairwise disjoint interior, and $M = \bigcup_i H_i$.
\item The intersection $H_{i} \cap H_{j}$ of any two of the handlebodies is a $3$--dimensional $1$--handlebody.
\item The common intersection $\Sigma = H_{0} \cap H_{1} \cap H_{2}$ of all three handlebodies is a closed, connected surface,  the \emph{central surface}.
\end{enumerate}
\end{Def}
The submanifolds   $H_{ij} = H_{i} \cap H_{j}$  and $\Sigma$ are referred to as the \emph{trisection submanifolds}.
In our illustrations, we use the colors blue, red, and green instead of $0$, $1$, and $2$ and we will refer to $H_{\textrm{blue}\thinspace \textrm{red}} = H_{br}$ as the \emph{green} submanifold and so on.

\begin{figure}[h]
\centering
\includegraphics[width=5cm]{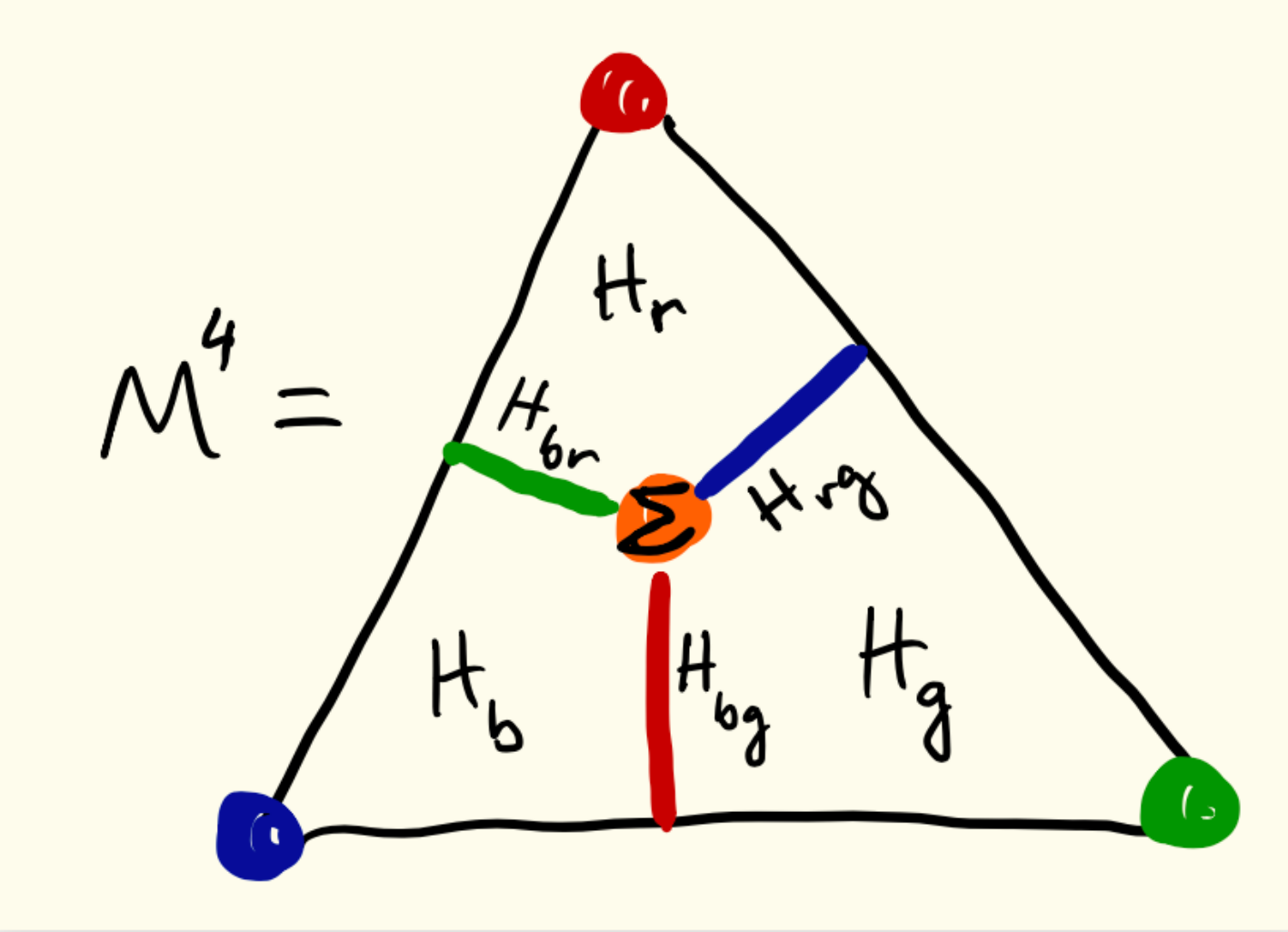}
\caption{Cartoon of a trisection.}
\label{fig:cartoon}
\end{figure}

The above definition is somewhat more general than the one originally given by Gay and Kirby~\cite{GK} in that they ask for the trisection to be \emph{balanced} in the sense that each handlebody $H_i$ has the same genus.
It was noted in \cite{RT2016} that any unbalanced trisection can be stabilized to a balanced one.
A representation of a trisection, dropping down two dimensions, is shown in Figure~\ref{fig:cartoon}.
This representation completely encapsulates our approach: We wish to define maps from $4$--manifolds to the $2$--simplex such that the \emph{dual cubical structure} of the $2$--simplex pulls back to trisections of the $4$--manifolds.

We use \emph{singular triangulations} to give a concrete description of a $4$--manifold $M$.
To induce a trisection of $M$ we use maps from $M$ to the standard $2$--simplex $\Delta^2$ that are induced by what we call \emph{tricolorings}.
The aim of this note is to describe an algorithm to compute a \emph{trisection diagram} on the central surface given an arbitrary singular triangulation of $M$, and to obtain complexity bounds on this description in terms of the size of the input triangulation.
The definitions are motivated by the example given in the next section and illustrated in Figure~\ref{fig:CP2}.


\section{Example}

Consider the moment map from the complex projective plane to the standard $2$--dimensional simplex, $\mu \from \mathbb{C}P^2 \to \Delta^2 \subset \mathbb{R}^3$ defined by
\[[\;z_0\; : \;z_1\; :\; z_2\;] \;\mapsto\; \frac{1}{\sum|z_k|}\;(\;|z_0|\;,\; |z_1|\;,\; |z_2|\;).\]
The \emph{dual spine} $\Pi^2$ in $\Delta^2$ is the subcomplex of the first barycentric subdivision of $\Delta^2$ spanned by the $0$--skeleton of the first barycentric subdivision minus the $0$--skeleton of $\Delta^2$.
Decomposing along $\Pi^2$ gives $\Delta^2$ a natural \emph{dual cubical structure} with three $2$--cubes, and the lower-dimensional cubes that we will focus on are the intersections of non-empty collections of these top-dimensional cubes, consisting of three \emph{interior $1$--cubes} and one \emph{interior $0$--cube}.
The cubical structure is indicated in Figure~\ref{fig:cartoon}, where the interior cubes are labeled with the trisection submanifolds.

\begin{figure}[ht]
\centering
\includegraphics[width=12cm]{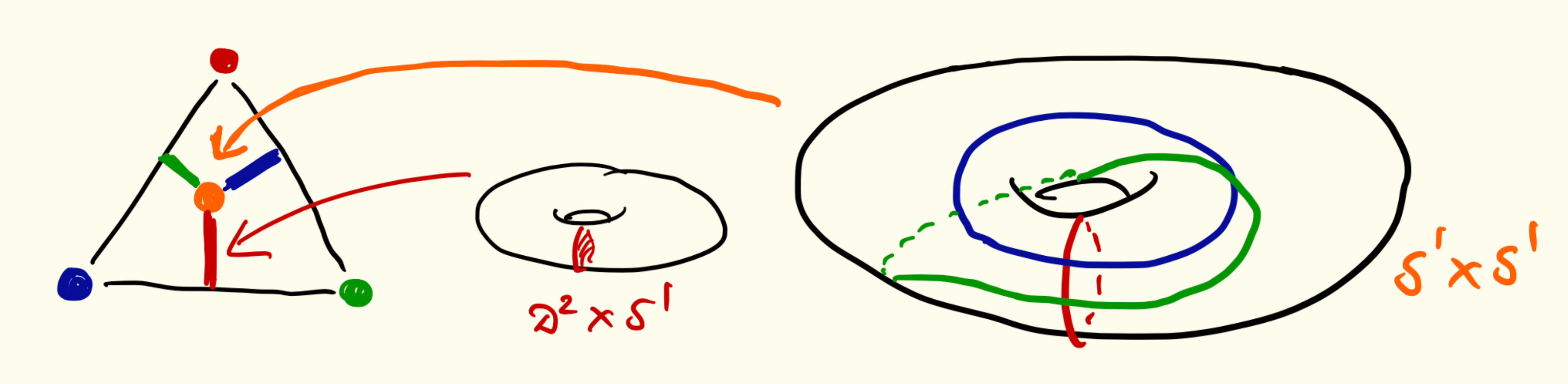}
\caption{Trisection diagram for $\mathbb{C}P^2$.}
\label{fig:CP2}
\end{figure}

Under the moment map, the $2$--cubes pull back to $4$--balls
\[\{\; [\;z_0\; : \;z_1\; :\; z_2\;] \;\mid\; z_i = 1, |z_j| \le 1, |z_k|\le 1\;\};\]
the interior $1$--cubes pull back to solid tori $S^1 \times D^2$ defined by
\[\{\; [\;z_0\; : \;z_1\; :\; z_2\;] \;\mid\; z_i = 1, |z_j| = 1, |z_k|\le 1\;\};\]
and the interior $0$--cube pulls back to a $2$--torus $\Sigma = S^1 \times S^1$ defined by
\[\{\; [\;z_0\; : \;z_1\; :\; z_2\;] \;\mid\; z_0 = 1, |z_1| = 1, |z_2| = 1\;\}.\]
The central surface is thus a Heegaard surface for the $3$--sphere boundary of each $4$--ball.
This shows that the cubical structure pulls back to a trisection with central surface a torus.
This is shown schematically in Figure~\ref{fig:CP2}.

Note that the midpoint of each edge of the $2$--simplex pulls back to a circle defined by
\[\{\; [\;z_0\; : \;z_1\; :\; z_2\;] \;\mid\; z_i = 0, z_j = 1, |z_k| = 1\;\}.\]
This is the core circle of the corresponding solid torus.
Each vertex of the $2$--simplex $\Delta^2$ pulls back to a singleton
\[\{\; [\;z_0\; : \;z_1\; :\; z_2\;] \;\mid\; z_i = 0, z_j = 0, z_k = 1\;\}.\]

There is more information in this picture.
The central surface is the pre-image of the barycenter of $\Delta^2$, and each solid torus is the preimage of the line segment joining this to the barycenter of a facet of $\Delta^2$.
This identifies the boundary curves of the three meridian discs.
Any two of these three curves give a Heegaard diagram for a $3$--sphere, and the union of all three curves is termed a \emph{trisection diagram} by Gay and Kirby~\cite{GK}.


\section{Constructing trisection diagrams}

In this section we define three notions of a tricoloring and describe an algorithm to compute trisections and trisection diagrams based on these colorings.
We will see that these colorings are readily constructed on triangulated $4$--manifolds.

\subsection{Singular triangulation}

Let $\widetilde{\Delta}$ be a finite union of pairwise disjoint, oriented Euclidean $4$--simplices with the standard simplicial structure.
Every $k$--simplex $\tau$ in $\widetilde{\Delta}$ is contained in a unique $4$--simplex $\sigma_\tau$.
A $3$--simplex in $\widetilde{\Delta}$ is termed a \emph{facet} and a $0$--simplex a \emph{vertex}.

Let $\Phi$ be a family of orientation-reversing affine isomorphisms pairing the facets in $\widetilde{\Delta}$, with the properties that $\varphi \in \Phi$ if and only if $\varphi^{-1} \in \Phi$, and every facet is the domain of a unique element of $\Phi$.
The elements of $\Phi$ are termed \emph{face pairings}.

We denote $\tri = (\widetilde{\Delta}, \Phi)$.
Any operation $O$ of simplicial topology that is performed on $\widetilde{\Delta}$ (such as barycentric subdivision, regular neighborhoods, and so on) is said to be an operation on $\tri$ so long as it respects the face pairings.
The set of all face pairings $\Phi$ determines a natural equivalence relation on the set of all $k$--simplices in $\widetilde{\Delta}$, and we will term the equivalence classes the \emph{(singular) $k$--simplices} of $\tri$.
This terminology is natural when passing to the quotient space $|\tri| = \widetilde{\Delta} / \Phi$ with the quotient topology.
The space $|\tri|$ is a closed, orientable $4$--dimensional pseudo-manifold, and the quotient map is denoted $p \from \widetilde{\Delta} \to |\tri|$.
The set of non-manifold points of $|\tri|$, if any, is contained in the $1$--skeleton.
(See Seifert--Threfall~\cite{SeiThr}.)

A \emph{singular triangulation} of a $4$--manifold $M$ is a PL homeomorphism $|\tri|\to M$, where $|\tri|$ is obtained as above.
The triangulation is \emph{simplicial} if $p \from \widetilde{\Delta} \to |\tri|$ is injective on each simplex.
The triangulation is \emph{PL} if, in addition, the link of every simplex is PL homeomorphic to a standard sphere.

\begin{figure*}
\centering
\includegraphics[width=.8\linewidth]{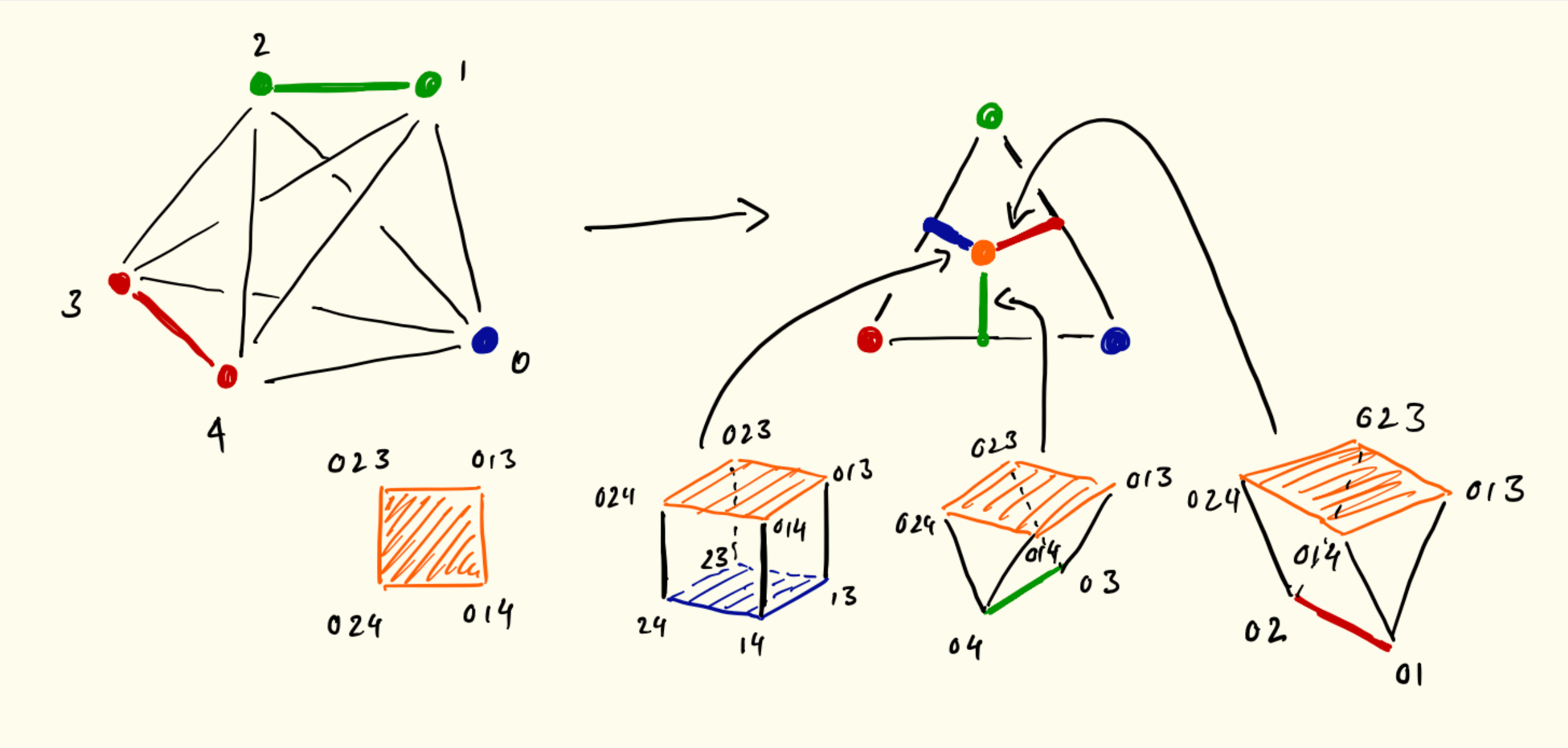}
\caption{Pieces of the trisection submanifolds. The vertices of the pieces are barycenters of faces and labeled with the corresponding vertex labels. The central surface meets the pentachoron in a square.
Two of the $3$--dimensional trisection submanifolds meet the pentachoron in triangular prisms and the third (corresponding to the singleton) meets it in a cube. Moreover, any two of these meet in the square of the central surface.}
\label{fig:pieces}
\end{figure*}

\subsection{Tricolorings}

Let $M$ be a closed, connected $4$--manifold with (possibly singular) triangulation $|\tri| \to M$.
A partition $\{P_0, P_1, P_2\}$ of the set of all vertices of $\tri$ is a \emph{tricoloring} if every $4$--simplex meets two of the partition sets in two vertices and the remaining partition set in a single vertex.
In this case, we also say that the triangulation is \emph{tricolored}.

Denote the vertices of the standard $2$--simplex $\Delta^2$ by $v_0$, $v_1$, and $v_2$.
A tricoloring determines a natural map $\mu \from M \to \Delta^2$ by sending the vertices in $P_k$ to $v_k$ and extending this map linearly over each simplex.
Note that the pre-image of $v_k$ is a graph $\Gamma_k$ in the $1$--skeleton spanned by the vertices in $P_k$.

As in the example of the complex projective plane, we would like to use $\mu$ to pull back the dual cubical structure of the simplex to a trisection of $M$.
The preimages of the dual cubes have very simple combinatorics.
The barycenter of $\Delta^2$ pulls back to exactly one $2$--cube in each pentachoron of $M$, and these glue together to form a surface $\Sigma$ in $M$.
This surface is the common boundary of each of the three $3$--manifolds obtained as preimages of an interior $1$--cube (edge) of $\Delta^2$.
Each such $3$--manifold is made up of cubes and triangular prisms, as in Figure~\ref{fig:pieces}.
Each interior $1$--cube $c$ has boundary the union of the barycenter of $\Delta^2$ and the barycenter $b$ of an edge of $\Delta^2$.
Since the map $\mu \from M \to \Delta^2$ is linear on each simplex, the preimage $\mu^{-1}(c)$ collapses to the preimage $\mu^{-1}(b)$.
In particular, each $3$--manifold has a spine made up of $1$--cubes and $2$--cubes.

Recall that a compact subpolyhedron $P$ in the interior of a manifold $M$ is called a (PL) spine of $M$ if $M$ collapses to $P.$ If $P$ is a spine of $M$, then $M \setminus P$ is PL homeomorphic with $\partial M \times [0,1).$ To see that the above construction gives a trisection, it suffices to show that:
\begin{enumerate}
\setlength\itemsep{0em}
\item the graph $\Gamma_k$ is connected for each $k$; and
\item the preimage of an interior $1$--cube of $\Delta^2$ has a $1$--dimensional spine.
\end{enumerate}

These conditions will be verified in the proof of the correctness of Construction~\ref{con:modify pseudo-tricoloring} below.
The first condition ensures that the preimage of each $2$--cube of $\Delta^2$ is a connected $4$--dimensional 1--handlebody. In particular, it has connected boundary. 
The second condition guarantees that the preimage under $\mu$ of each interior $1$--cube of $\Delta^2$ is a $3$--dimensional handlebody, hence also has connected boundary. Since the boundary of the $4$--dimensional 1--handlebody is the union of such handlebodies, this implies that the central surface is connected.

We say that a tricoloring is a \emph{c-tricoloring} if $\Gamma_k$ is connected for each $k$ and that a c-tricoloring is a \emph{ts-tricoloring} if the preimage of each interior $1$--cube collapses onto a $1$--dimensional spine.
In this case, the dual cubical structure of $\Delta^2$ pulls back to a trisection of $M$.

For example, the standard $4$--sphere $S^4$ can be thought of as a doubled $4$--simplex, giving it a singular triangulation with two $4$--simplices and five vertices, denoted $v_0, \ldots, v_4$.
Letting $P_0 = \{v_0, v_1\}$, $P_1 = \{v_2, v_3\}$, and $P_2 = \{v_4\}$ gives a ts-tricolored triangulation of $S^4$ with each of $\Gamma_0$ and $\Gamma_1$ a $1$--simplex and $\Gamma_2$ a $0$--simplex.

Examples of c-tricolorings, which are not ts-tricolorings, can be constructed so that the preimage of the dual cubical structure of $\Delta^2$ is given by gluing together three $4$--dimensional handlebodies $H_i$, so that $H_1 \cap H_2 \cap H_3$ is not a Heegaard surface for each $H_i \cap H_j$.

\subsection{Existence of tricolorings}

In general, given an arbitrary triangulated $4$--manifold $M$, one can always obtain a tricolored triangulation by passing to the first barycentric subdivision.
This has a natural partition of the vertices into five sets $B_i$ (the barycenters of the $k$--simplices for $0 \le k \le 4$).
Any coarsening of this partition of the form $\{B_i \cup B_j, B_k \cup B_l, B_m\}$ now gives a tricoloring.
For instance, the partition $\{B_0 \cup B_1, B_2 \cup B_3, B_4\}$ was used in \cite{RT2016}.
While conceptually simple, this process multiplies the number of $4$--simplices by a factor of $120$.
We now give an improved construction.

\begin{figure*}[t]
\centering
\includegraphics[width=.8\linewidth]{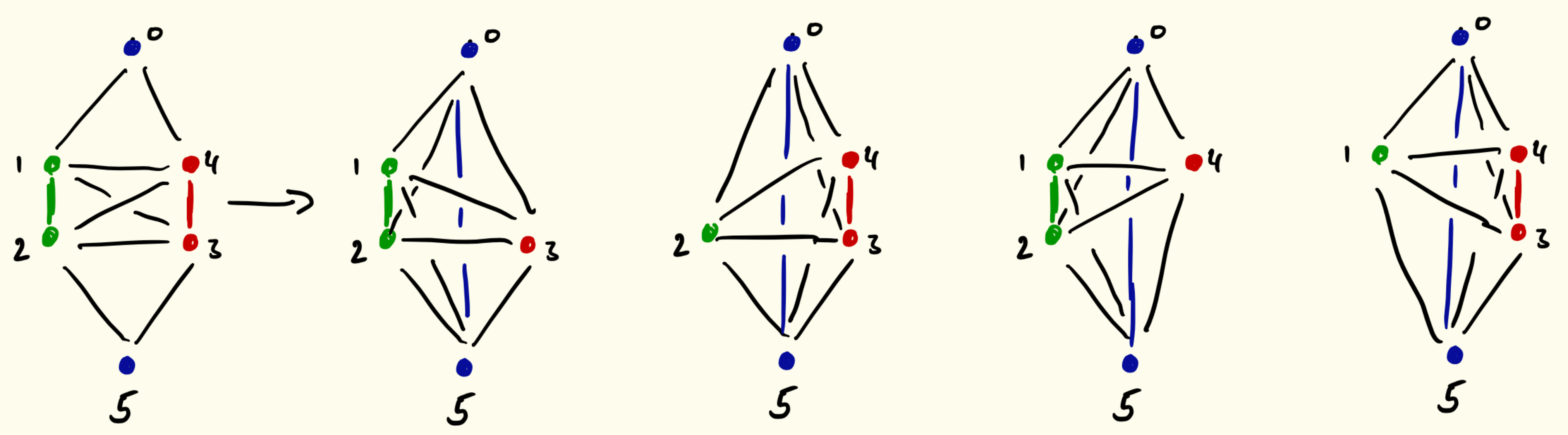}

\vspace{0.3cm}
\includegraphics[height=.12\linewidth]{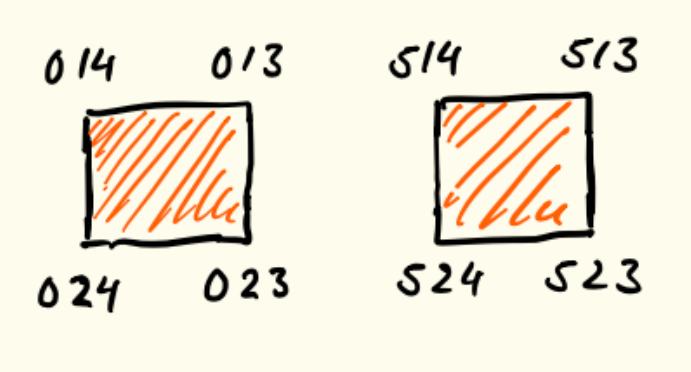}
\hspace{0.2cm}
\includegraphics[height=.12\linewidth]{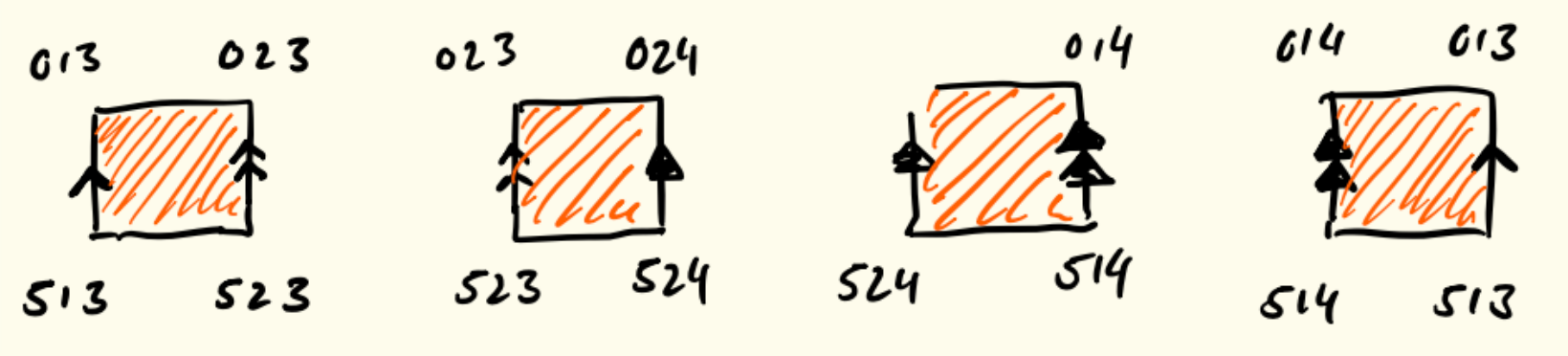}
\caption{Converting a double-pentachoron to a quadra-pentachoron. Shown are also the squares of the central submanifold --- the move replaces two disjoint discs with an annulus, thus adding a handle to the central surface. The remaining figures are base on this; the reader should bear in mind that there are two other kinds of pentachoron corresponding to permutations of the colors.}
\label{fig:2-4-move}
\end{figure*}

\begin{Con}
\label{con:construct pseudo-tricoloring}
Given an arbitrary triangulation of the closed $4$--manifold $M$ having $n$ pentachora, there is a triangulation with $60n$ pentachora that admits a tricoloring.
Moreover two of the graphs $\Gamma_k$ are connected and the third consists of isolated vertices.
\end{Con}

\begin{proof}
Let $|\tri| \to M$ be a (possibly singular) triangulation of $M$.
Each $4$--simplex $\sigma$ in $\tri$ has a natural subdivision into sixty $4$--simplices with set of vertices consisting of its $0$--simplices together with its barycenter and all barycenters of its $2$--simplices and $3$--simplices.
This can be built by first applying a 1--5 bistellar move to $\sigma$.
Each new pentachoron $\sigma' < \sigma$ has a unique $3$--face that corresponds to a $3$--simplex in $\sigma$.
Perform a 1--4 move on this $3$--simplex and cone this to a triangulation of $\sigma'$ consisting of four $4$--simplices.
Each of the resulting $4$--simplices $\sigma''$ contains a unique $2$--simplex that corresponds to a $2$--simplex in $\sigma$.
Perform a 1--3 move on this $2$--simplex and cone this to a triangulation of $\sigma''$.

For each $4$--simplex $\sigma$ in the resulting triangulation there is a flag $\sigma^1 < \sigma^2 < \sigma^3 < \sigma^4$ of simplices    such  that the vertices of $\sigma$ consist of the vertices of $\sigma^1$ and the barycenters of $\sigma^2$, $\sigma^3$, and $\sigma^4$.
The tricoloring is now obtained by placing the vertices of $\sigma^1$ in the set $P_0$, the barycenter of $\sigma^2$ in the set $P_1$ and the barycenters of $\sigma^3$ and $\sigma^4$ in the set $P_2$.
This gives a tricoloring with $60n$ pentachora.

It is easy to see that $\Gamma_0, \Gamma_2$ are connected, as they are   the    $1$--skeleton and dual $1$--skeleton of the original triangulation.
Furthermore $\Gamma_1$ consists of isolated vertices since each pentachoron has exactly one vertex in the set $P_1$.
\end{proof}

\subsection{From tricolorings to ts-tricolorings}

We now show that given any tricolored triangulation of $M$, there is a simple procedure that transforms it into a ts-tricolored triangulation of $M$.

\begin{Con}
\label{con:modify pseudo-tricoloring}
Given a tricolored triangulation of the closed $4$--manifold $M$ with $n$ pentachora, there is a ts-tricolored triangulation with $2n$ pentachora.
\end{Con}

\begin{proof}
In a tricolored triangulation, each pentachoron $\sigma$ has a unique facet $\tau$ that meets only two of the partition sets.
There is a unique pentachoron $\sigma'$ meeting $\sigma$ in $\tau$, and the two together form
 a \emph{double-pentachoron}, $\sigma \cup_\tau \sigma'$.
The manifold $M$ thus has a decomposition into double-pentachora\footnote{To the knowledge of the authors, this elementary fact has not been observed previously.}.
There are three types of double-pentachora, classified by the isolated vertices.
Shown in Figure~\ref{fig:2-4-move} is a double-pentachoron with vertices numbered $1$ to $4$ in $\tau$ (drawn in green and red) and $0$ and $5$ not in $\tau$ (drawn in blue).
The three colors correspond to the partition sets $P_k$, and vertices of the same color may be identified in $M$.
Throughout this proof, we may without loss of generality refer to this labeled double-pentachoron.
The other two types arise by permuting the three colors.

We focus on one of the partition sets, $P_k$.
The graph $\Gamma_k$ meets a double-pentachoron either in a single edge or in two isolated vertices.
Perform a 2--4 move on each of the double pyramids meeting $\Gamma_k$ in two isolated vertices.
This gives a new triangulation $\tri'$ and a new graph $\Gamma'_k$.
Each pentachoron in $\tri'$ meets $\Gamma'_k$ in an edge, and hence the graph $\Gamma'_k$ is connected.
To see that $\Gamma'_k$ is connected, choose any two vertices in $P_k$, contained in two pentachora of the triangulation.
A path in the dual $1$--skeleton between barycenters of these pentachora can be deformed into $\Gamma'_k$ after the 2--4 moves, showing that this graph is indeed connected.

Notice that $\Gamma'_k$ is obtained from $\Gamma_k$ by adding one edge for each double pyramid on which a 2--4 move was performed.
See Figure~\ref{fig:2-4-move}, where vertices in $P_k$ are drawn in blue.
This does not affect any of the other monochromatic graphs.
Since this can be done independently for each $k$ (adding edges does not change the other graphs), this shows that after doing all 2--4 moves, we have a c-tricoloring.

We claim that in fact we also have a ts-tricoloring after performing all 2--4 moves.
This has to do with special properties of the degree four edges obtained in doing these moves.
Let us introduce some additional terminology that will be useful.
A 2--4 move performed on a double-pentachoron gives a \emph{quadra-pentachoron}; that is, a collection of four pentachora meeting in a common  $1$--simplex   contained in no other pentachora, and with a particular coloring having two vertices of each color.
This structure of the quadra-pentachora is crucial to our constructions and proofs.

Let $Q$ be a collection of four pentachora forming one of these quadra-pentachora.
The boundary of $Q$ consists of eight tetrahedra and $M$ is tiled by disjoint collections of these quadra--pentachora, meeting along common tetrahedral faces.
After the collection of 2--4 moves:
\begin{enumerate}
\setlength\itemsep{0em}
\item The monochromatic subgraph of the $1$--skeleton $\Gamma_k$ is connected for each $k$.
\item $\Sigma \cap Q$ is an annulus formed from squares, one square in each of the four pentachora of $Q$.
The boundary of this annulus consists of eight edges lying in eight boundary tetrahedra of $Q$, with eight vertices lying in eight boundary $2$--simplices of $Q$.
These combine to give a decomposition of $\Sigma$ into annuli.
\item For each $k$ and $3$--dimensional trisection submanifold $H_{ij}$, the intersection $H_{ij} \cap Q$ consists of one of two types of polyhedral structures.

The first polyhedral structure is a $3$--ball $B \subset Q$ whose boundary is tiled by eight square and four triangular faces.
Four of the square faces form an annulus that lies on $\Sigma$ and the others form a pair of disks lying in $\partial Q$.
There is a  \ collapse of $B$ to a $1$--dimensional spine (an ``H'') and on $\partial Q$ this collapse agrees with those defined on adjacent quadra--pentachora.
This follows from the facet pairings indicated in Figure~\ref{fig:pieces-r}.
For example, consider the red square that the red submanifold $H_{bg}$  collapses  to, shown in Figure~\ref{fig:pieces-r}.
This square has two edges, $(01, 51)$ and $(02, 52)$, that lie in the interior of $Q$ and thus are not glued to any other red squares.
Thus the collapse to an ``H'' in $Q$ matches with similar collapses in adjacent quadra-pentachora.

\begin{figure}[t]
\centering{
\includegraphics[width=10cm]{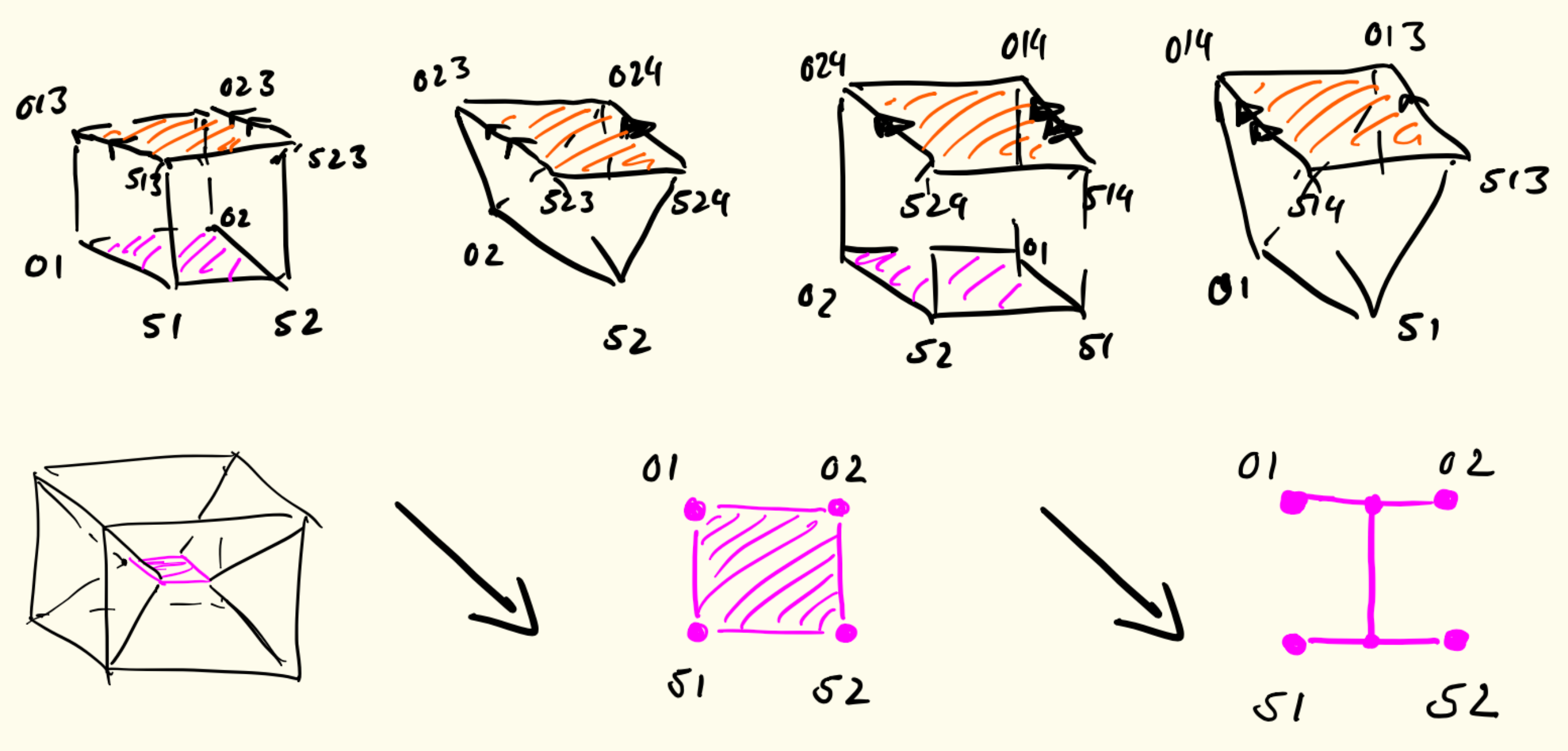}

\vspace{0.2cm}

\includegraphics[width=10cm]{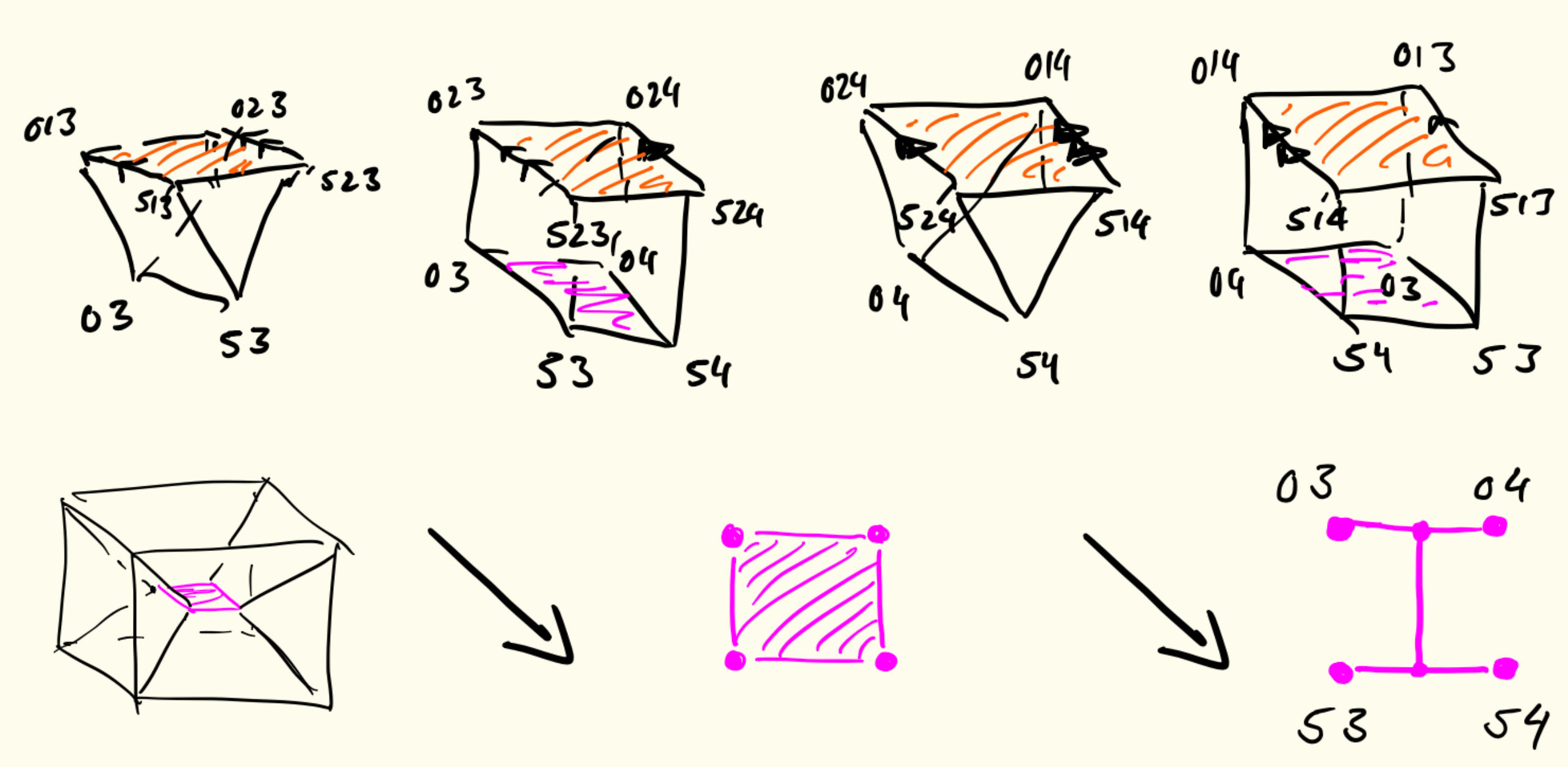}
}
\caption{Red and green submanifold (based on Fig.~\ref{fig:2-4-move}). Blocks that form pieces of trisection submanifolds. The vertices of the blocks are barycenters of faces of the triangulation, labeled with the corresponding vertex labels. The picture for the red and green submanifolds are analogous with the notable difference that the green submanifold meets the pentachora in cubes (resp. prisms) that the red submanifold meets in prisms (resp. cubes).}
\label{fig:pieces-r}
\end{figure}

The second polyhedral structure is a solid torus $T \homeo S^1 \times D^2$ whose boundary is tiled by twelve square faces.
The boundary of $T$ intersects $\partial Q$ in eight of these square faces, two in each pentachoron of $Q$.
This solid torus   collapses  to a curve consisting of four line segments, and on $\partial Q$ this collapse agrees with those on adjacent quadra-pentachora.
See Figure~\ref{fig:pieces-b}.

\begin{figure}[t]
\centering
\includegraphics[width=10cm]{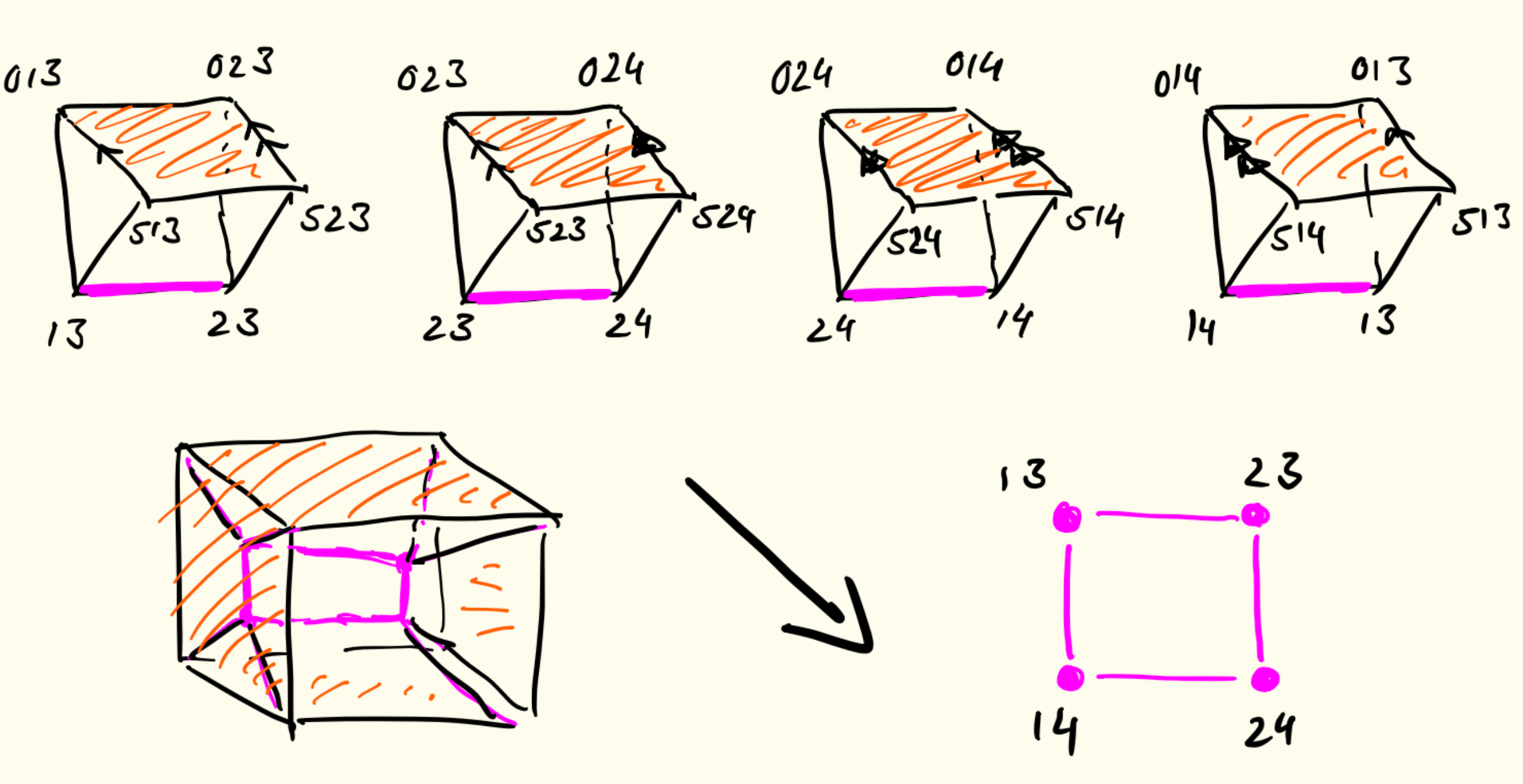}
\caption{Blue submanifold (based on Fig.~\ref{fig:2-4-move}). Blocks that form pieces of trisection submanifolds. The vertices of the blocks are barycenters of faces of the triangulation, labeled with the corresponding vertex labels.}
\label{fig:pieces-b}
\end{figure}

\item $X_k \cap Q$ is a $4$--ball that   collapse  to $ \Gamma'_k \cap Q$ and restricts on $\partial Q$ to a   collapse  of $X_k \cap \partial Q$ that agrees with those on adjacent quadra--pentachora.
\end{enumerate}

This discussion shows that each pairwise intersection of $4$--dimensional handlebodies is indeed a $3$--dimensional handlebody, since it has a $1$--dimensional spine, and hence completes the proof.
\end{proof}

Combining Constructions~\ref{con:construct pseudo-tricoloring} and \ref{con:modify pseudo-tricoloring} gives:

\begin{Thm}
\label{thm:ts-coloring}
Given an arbitrary triangulation of the closed orientable $4$--manifold $M$ having $n$ pentachora, there is a triangulation with $120n$ pentachora that admits a ts-tricoloring.
\end{Thm}


\subsection{Constructing trisection diagrams}

The proof of Construction~\ref{con:modify pseudo-tricoloring} allows the construction of the compression discs of the $3$--dimensional handlebodies (and hence the trisection diagram) from a ts-tricolored triangulation with a decomposition into quadrapentachora.
The details will now be given.

The $3$--dimensional $1$--handlebodies $H_{ij}$ have a coarse decomposition into polyhedral balls and polyhedral solid tori, and a finer decomposition of each polyhedral ball into two $3$--cubes and two triangular prisms, and of each polyhedral torus into four triangular prisms.
The initial spine for each $H_{ij}$ consisted of a $2$--cube in each polyhedral ball and of a circle consisting of four $1$--cubes in each polyhedral solid torus.
It was then shown that each $2$--cube can be collapsed further to an ``H''.
In order to analyze the spine further, we use the natural simplicial subdivision of an ``H'' into five $1$--simplices.
Note that a square face of a polyhedral cube may glue to a square face of a polyhedral solid torus.
We therefore also subdivide the $1$--cubes in the polyhedral solid tori into two $1$--simplices.
This gives a consistent subdivision of the $1$--dimensional spine of $H_{ij}$ (possibly with some redundancies that can be avoided in an efficient implementation).

\begin{figure}[ht]
\centering
\includegraphics[height=3cm]{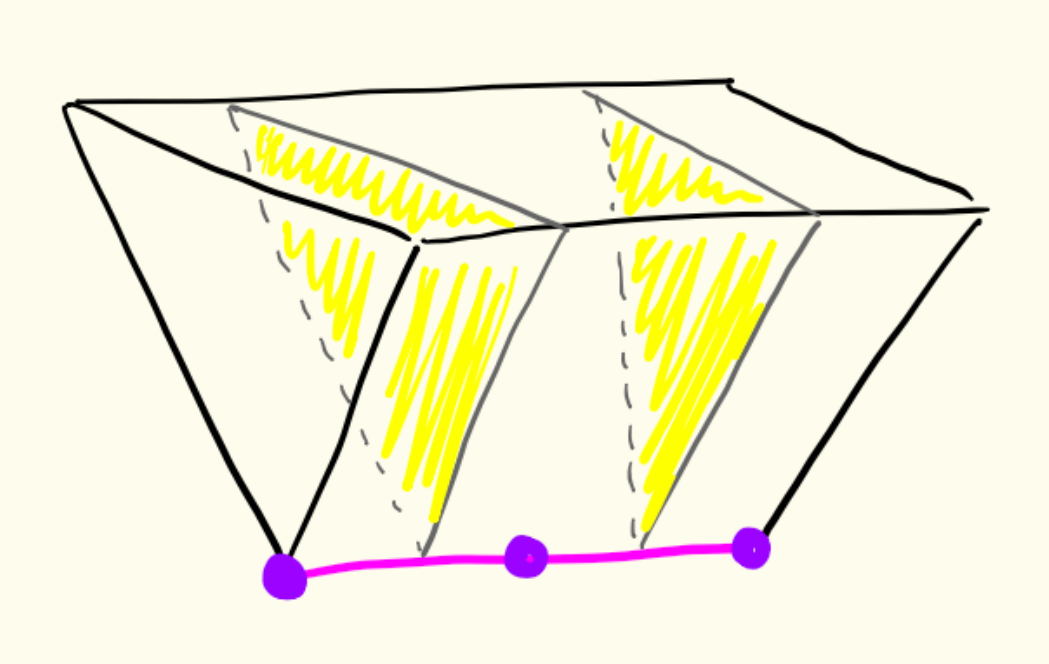}
\hspace{1cm}
\includegraphics[height=3cm]{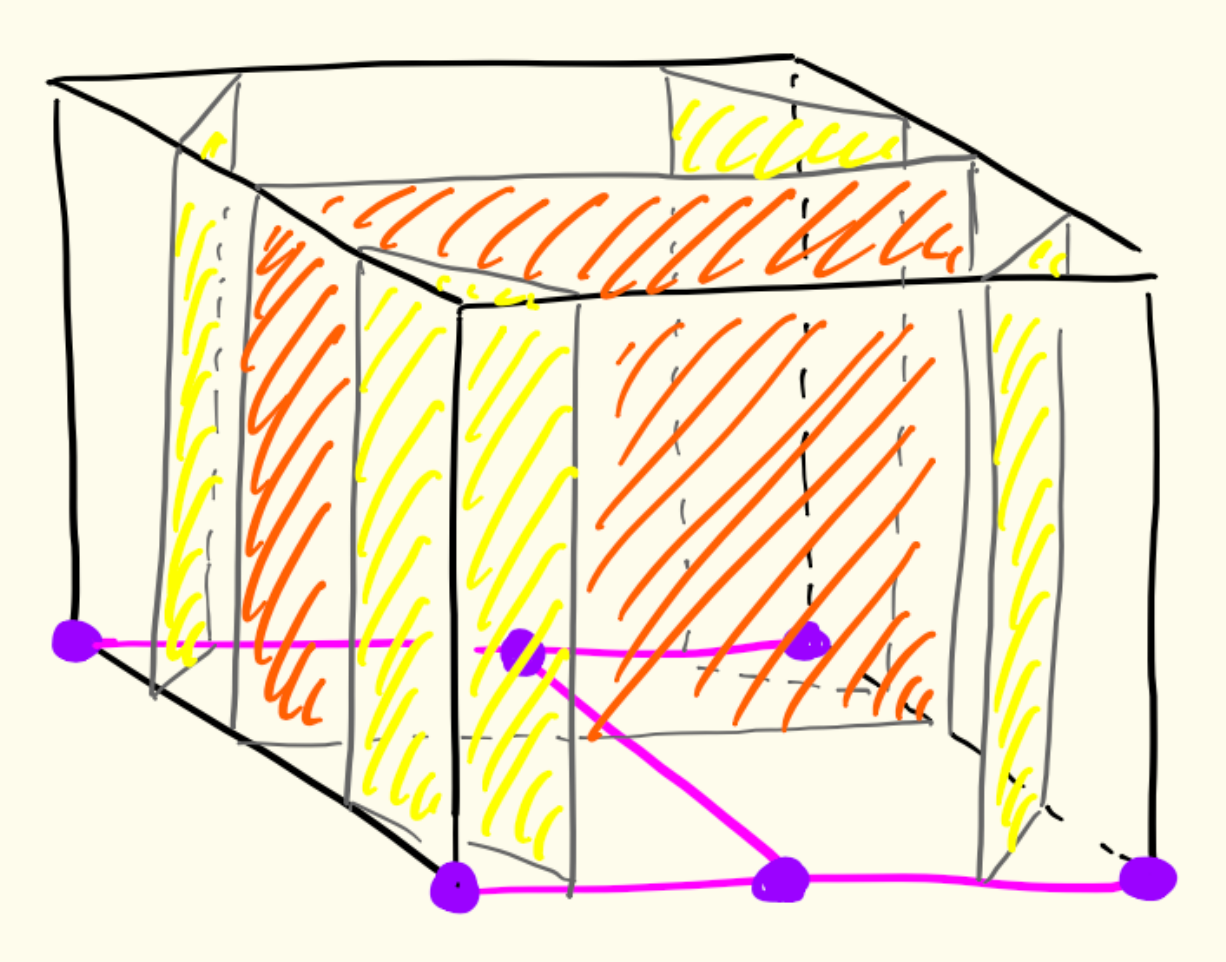}
\hspace{1cm}
\includegraphics[height=3cm]{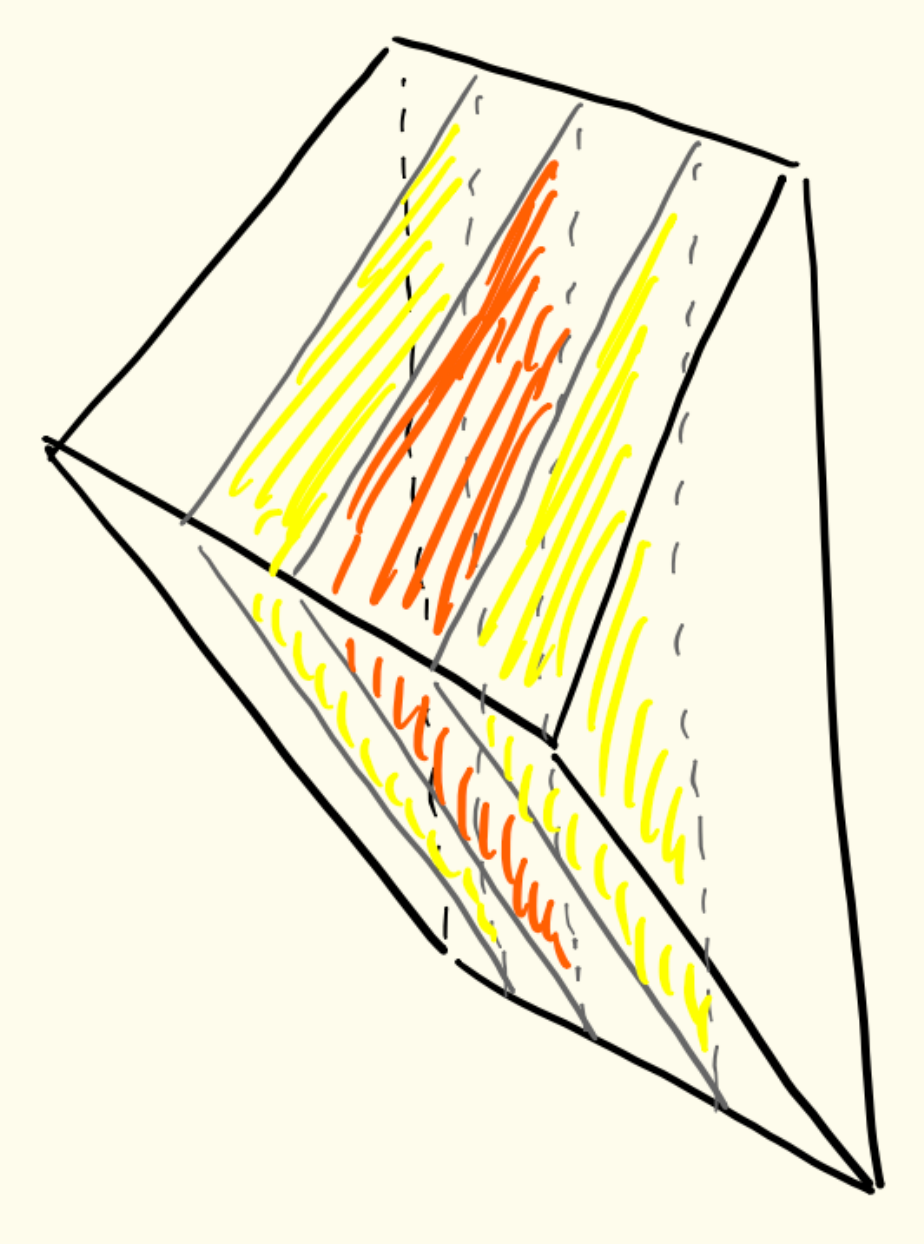}
\caption{Canonical triangles and squares in blocks; from left to right: prism in torus, cube, and prism in ball.}
\label{fig:meridian_discs_0}
\end{figure}

We claim that for each handlebody, a complete system of compression discs is constructed, by adding canonical normal squares and normal triangles as show in Figure~\ref{fig:meridian_discs_0}.
In a polyhedral solid torus, there are two normal triangles in each triangular prisms which are dual to the two $1$--simplices of the spine.
In a polyhedral solid torus, there is one central square in each cube that is dual to the internal edge of the ``H'', and each edge meeting a boundary vertex of the ``H'' has a correponding dual square.
In order to give well defined discs in the polyhedral balls, we introduce three normal triangles in each triangular prism contained in it.

The surface formed in each of the polyhedral structures is shown in Figure~\ref{fig:meridian_discs_1}.
This directly shows that the surface meeting the internal edge of the ``H'' is a disc transverse to the spine, and hence a meridian disc.

\begin{figure}[ht]
\centering
\includegraphics[height=3cm]{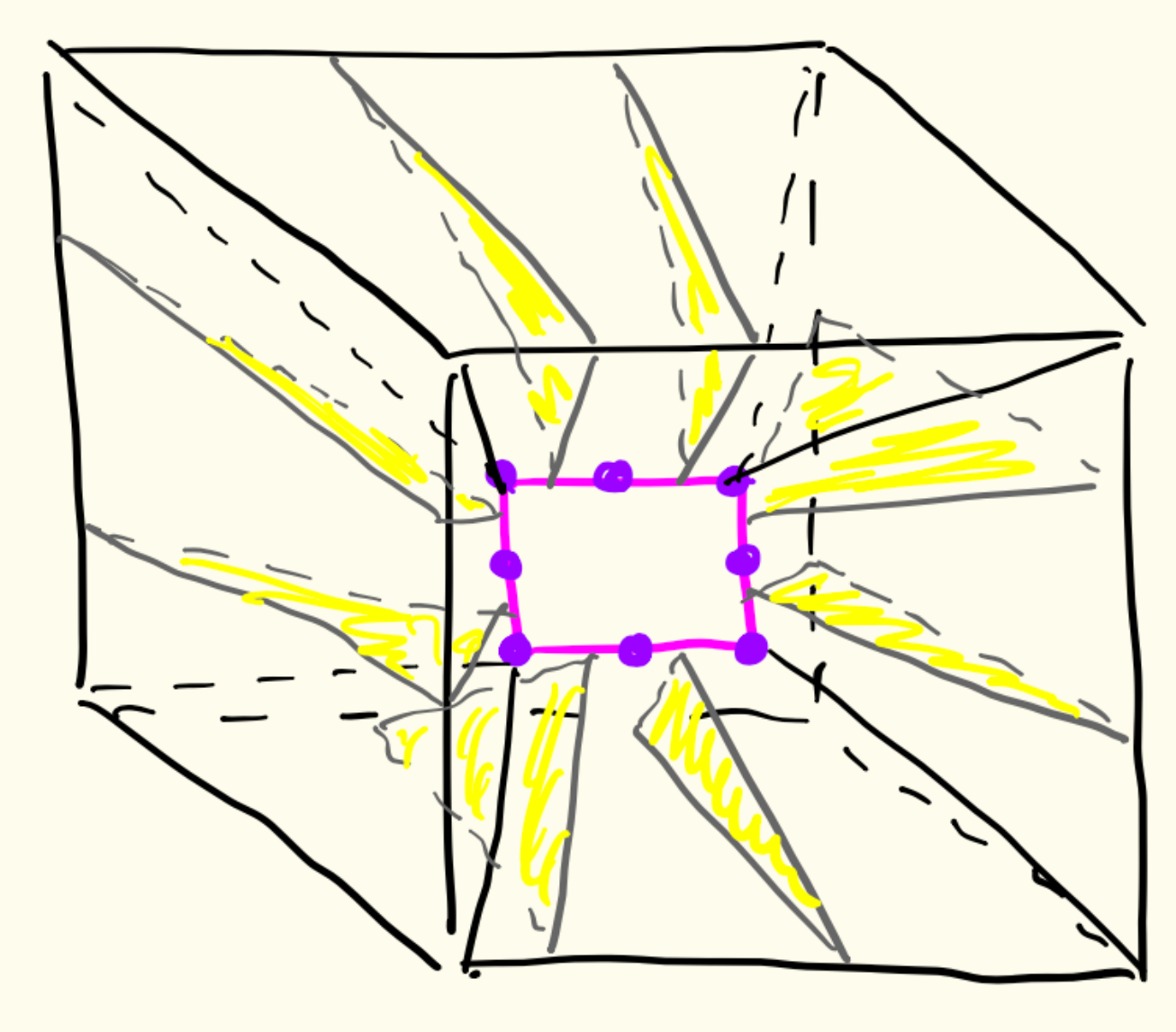}
\hspace{1cm}
\includegraphics[height=3cm]{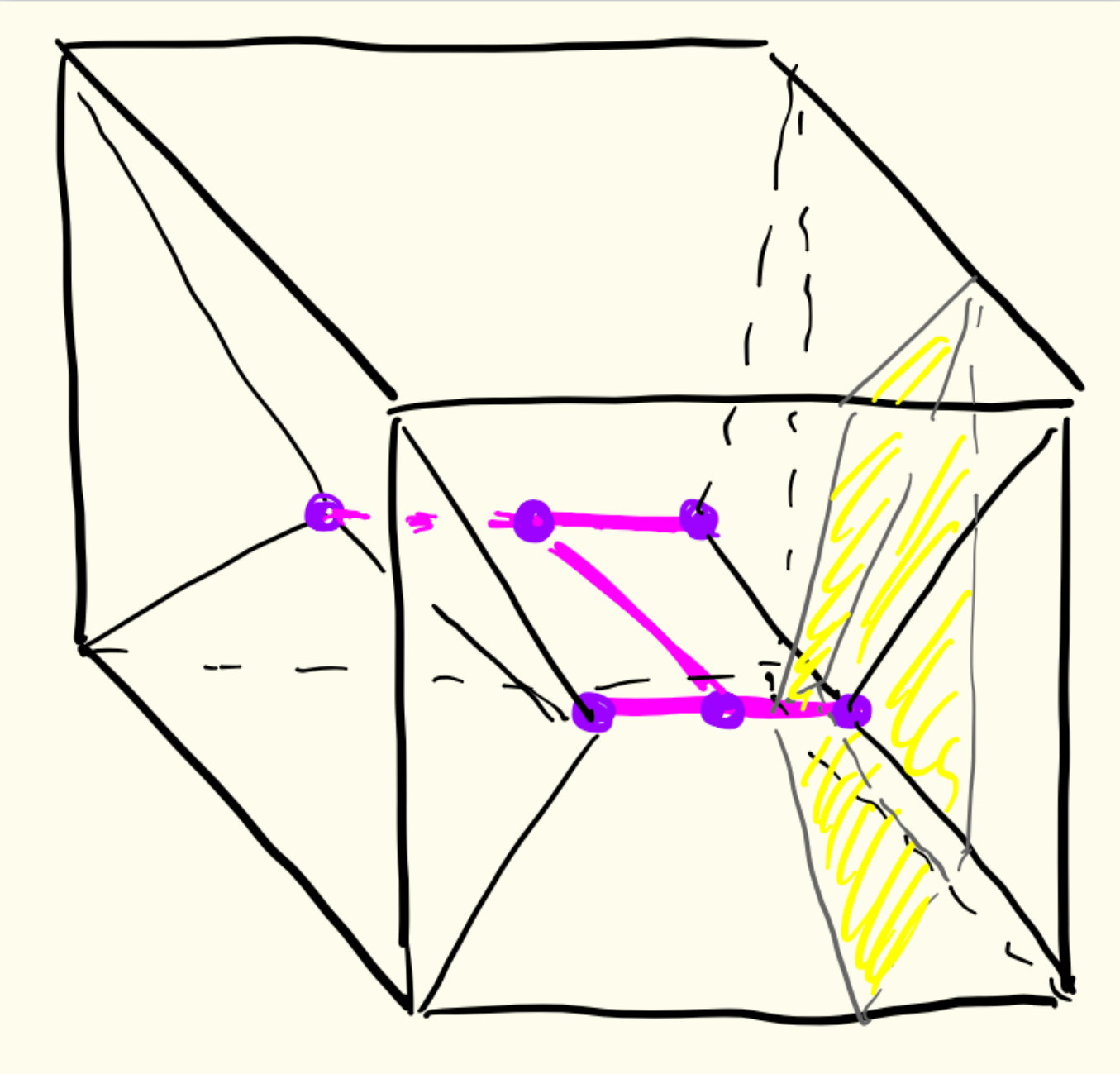}
\hspace{1cm}
\includegraphics[height=3cm]{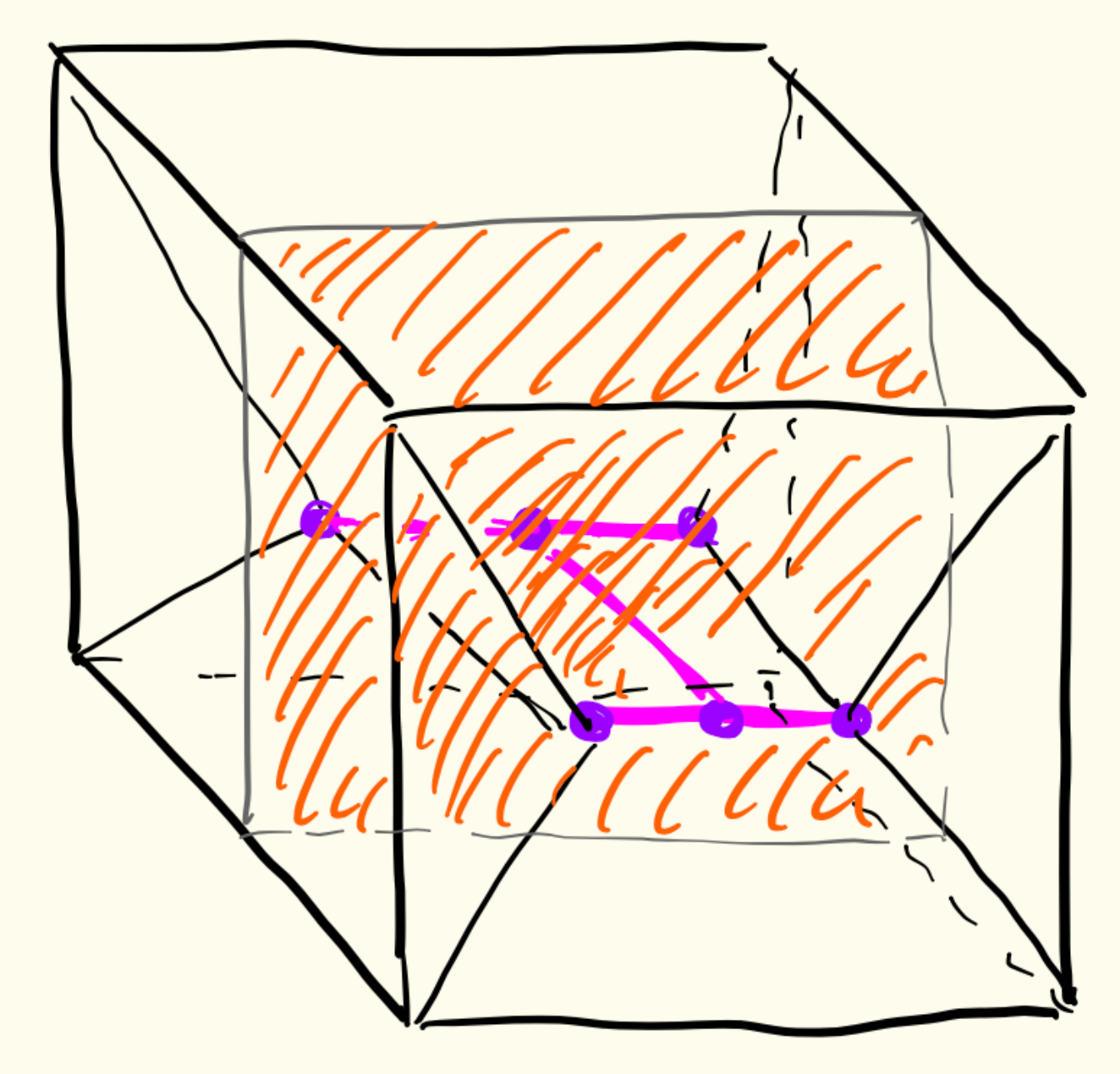}
\caption{Canonical triangles and squares in the polyhedral structures of the submanifolds: all discs in the torus structure (left); one of the four discs parallel to vertical edges in the cube structure (middle) and the central discs in the cube structure (right).}
\label{fig:meridian_discs_1}
\end{figure}

For the remaining discs, we need to show that no branching occurs along the edges of the normal triangles and quadrilaterals.
It follows from the labeling of barycenters that such an edge lies in a $3$--simplex in the triangulation, and hence meets at most two building blocks of the $3$--dimensional handlebody.
Hence the surface is properly embedded.
The claim that each component is a disc now follows from the fact that each triangle meets a unique $1$--dimensional stratum of the spine in a vertex.
Developing this surface normally to the spine around this central vertex can only give a disc.

To see the discs form a complete system (possibly with redundancies), note that each edge of the spine has a dual disc.
The central surface $\Sigma$ is decomposed along a graph into annuli.
Each annulus is made up of four squares, giving the surface a natural singular Euclidean structure.
Each such annulus is met in a single core curve and in two pairs of boundary parallel arcs by two of the three sets of meridian curves.
The remaining set meets the core curve transversely in eight essential arcs, one in each square.
This is shown in Figure~\ref{fig:diagram}.

\begin{figure}[ht]
\centering
\includegraphics[width=7cm]{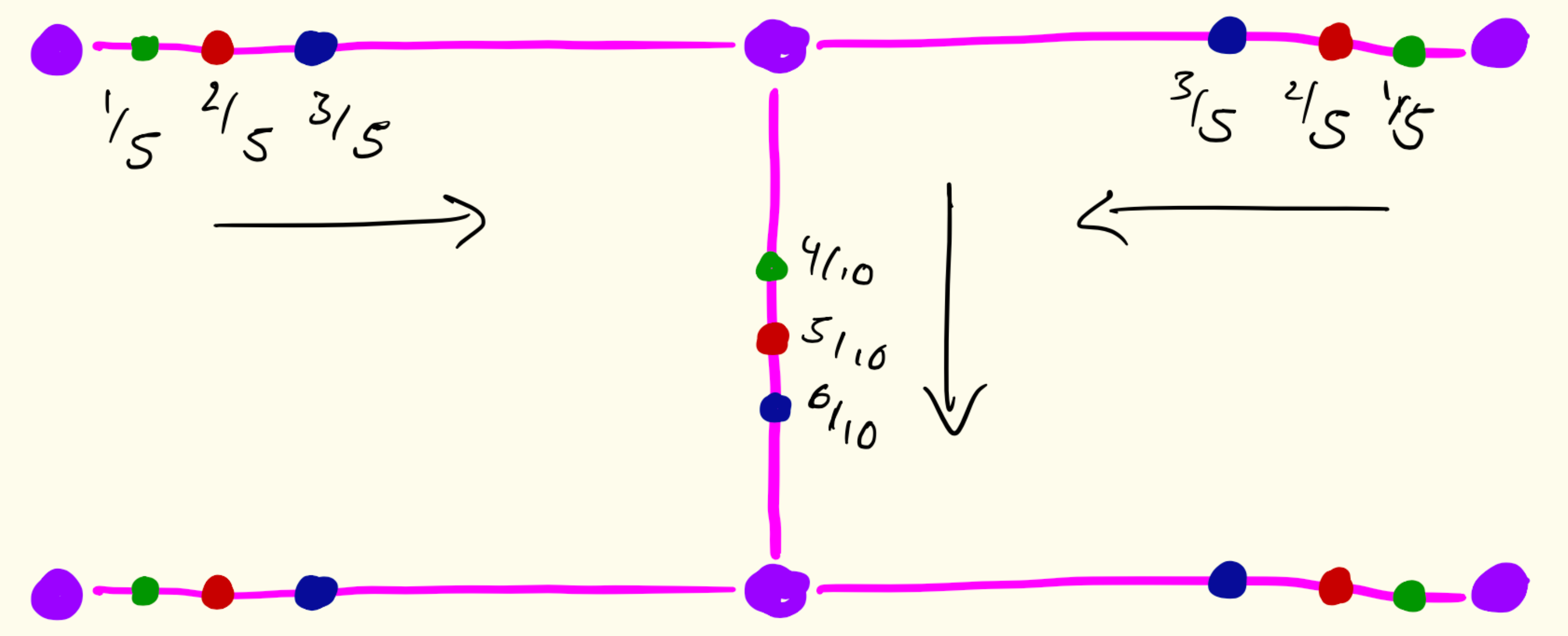}
\quad
\includegraphics[width=7cm]{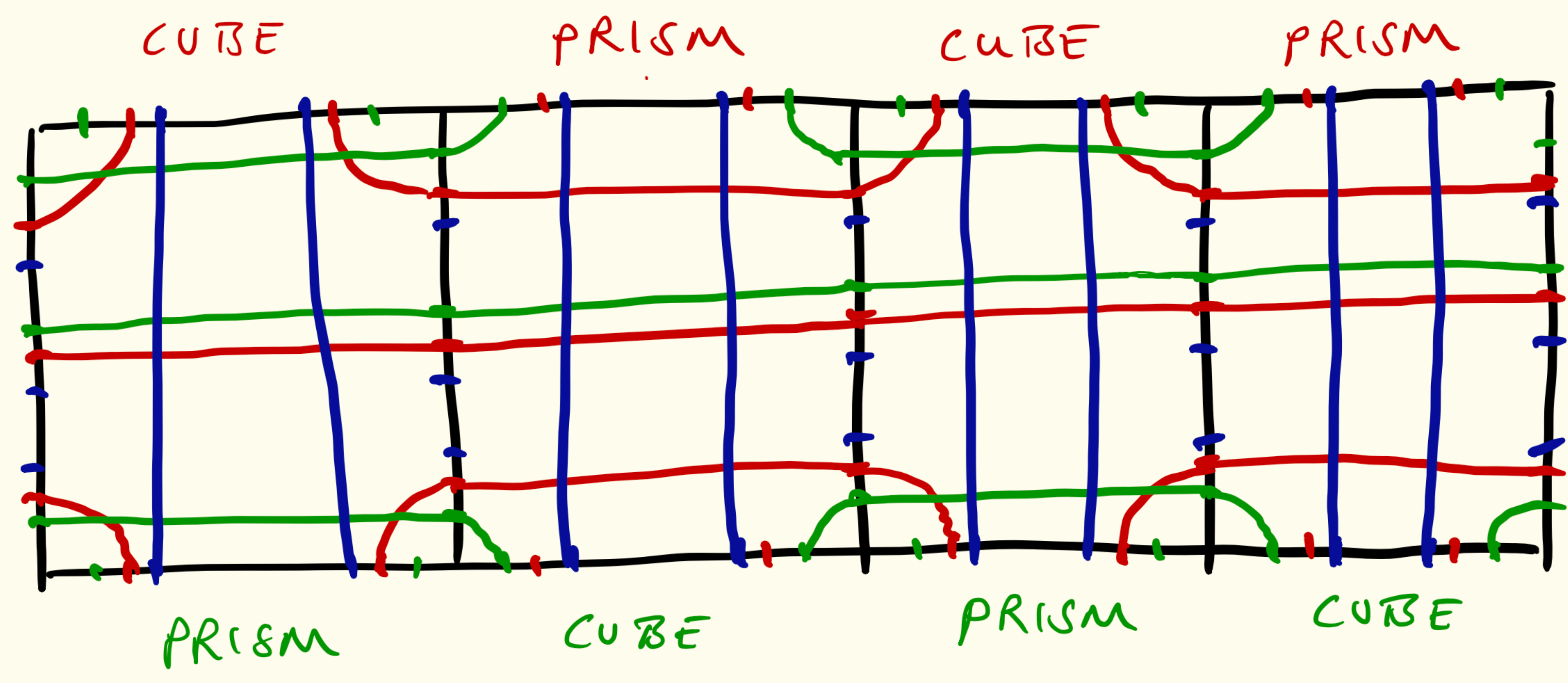}
\caption{The trisection diagram on a cubulated annulus of $\Sigma$: the top shows the placement of the intersection points with the spine; the bottom picture shows the curves on the cubulated annulus consistent with Figure~\ref{fig:2-4-move}. Indicated is also that whenever the ball structure of the red submanifold meets a square of the central surface in a cube, then the ball structure of the green submanifold meets it in a prism and vice versa.}
\label{fig:diagram}
\end{figure}

One can make the placement of the discs in the blocks completely canonical as follows.
Each $1$--simplex that is in a subdivided $1$--cube of the spine is oriented towards the midpoint of the $1$--cube.
One then marks the potential intersection points of normal discs with these $1$--simplices as $\frac{1}{5}$ from the initial vertex for $H_{12}$, $\frac{2}{5}$ for $H_{02}$, and $\frac{3}{5}$ for $H_{12}$.
Each internal edge of an ``H'' is oriented arbitrarily and the potential intersection points are marked at $\frac{4}{10}$, $\frac{5}{10}$, and $\frac{10}{10}$ respectively.
These placements can now be extended linearly over the cubes and prisms and give markings on the squares of the central surface.
The resulting curves are hence transverse.
After a clean-up step that removes parallel copies of curves, we obtain the desired trisection diagram.

\section{Complexity bounds}

In this section, the aim is to give a bound for the genus of the central trisection surface in terms of the number of $4$--simplices in a triangulation of a $4$--manifold.

\begin{Thm}
\label{thm:genus}
Suppose $M$ is a closed orientable $4$--manifold with a triangulation having a ts-tricoloring and $n$ pentachora.
Then the genus of the central trisection surface is at most  $n/2$.
\end{Thm}

\begin{proof}
Every pentachoron contributes one quadrilateral to the central trisection surface $\Sigma$.
Hence $\Sigma$ has a quadrangulation with $n$ quadrilaterals, $2n$ edges, and at least one vertex.
The Euler characteristic satisfies $\chi(\Sigma) \ge 1-n$   and this implies that  $\genus(\Sigma) \le  (n+1)/2 $, since $n$ is even.
\end{proof}

\begin{Coro}
\label{cor:genus bound arbitrary}
Suppose $M$ is a closed orientable $4$--manifold with an arbitrary triangulation with $m$ pentachora.
Then there is a trisection with central surface having genus at most $60m$.
\end{Coro}

\begin{proof}
This follows by combining the results of Theorem~\ref{thm:ts-coloring} and Theorem~\ref{thm:genus}.
\end{proof}

Since a surgery description of a $4$--manifold $M$ can be converted into a triangulation, this implies an upper bound for the genus of a trisection of $M$, with input a Kirby diagram. It would be interesting to determine explicit complexity bounds on the trisection genus from different descriptions of $4$--manifolds, such as Kirby diagrams.

\begin{Que}
Can one find general bounds that are asymptotically sharp for infinite families of examples?
\end{Que}


\section{Moves simplifying tricolored triangulations}

Given a c-tricolored (respectively ts-tricolored) triangulation, we give a criterion to obtain a \emph{collapsed} triangulation that is also c-tricolored (respectively ts-tricolored) but has only three vertices, one for each color. We also show that tricolored triangulations can be simplified with some Pachner moves whilst maintaining the coloring property.

\subsection{Tricolored triangulations with few vertices}

An edge $E$ of a triangulation $\mathcal{T}$ of a  $4$--manifold is contained in a bubble $2$--sphere $S$ if the following three conditions are satisfied: 
\begin{enumerate}
\setlength\itemsep{0em}
\item  There is an even collection of singular 2--simplices $F_1, F_2, \dots , F_{2k}$ of $\mathcal{T},$ each containing the edge $E$.
\item For each $i$ the remaining edges of $F_i$ can be labelled $E_i^-$ and $E_i^+$, such that $E_i^+ = E_{i+1}^-$ and $E_{2k}^+=E_1^-.$
\item If there is a tricolouring of the triangulation, then $E$ is a monochromatic edge with ends on two different vertices. 
\end{enumerate}

We say that $\tri^\star = (\widetilde{\Delta}^\star, \Phi^\star)$ is obtained from $\tri = (\widetilde{\Delta}, \Phi)$ by \emph{collapsing} the edge $E$ if $\widetilde{\Delta}^\star$ consists of all pentachora in $\widetilde{\Delta}$ not containing the edge $E$ and the face pairings $\Phi^\star$ are obtained as follows. Each facet $\tau$ in $\widetilde{\Delta}^\star$ is the domain of a unique $\varphi_\tau \in \Phi.$ If the codomain of $\varphi_\tau$ is also a facet in $\widetilde{\Delta}^\star,$ then $\varphi_\tau \in \Phi^\star.$ Otherwise the codomain of $\varphi_\tau$ is a facet $\tau_1$ of a pentachoron $\sigma_1$ containing $E$ and the collapse $\chi$ of $\sigma_1$ naturally identifies this with another facet $\tau_2$ of $\sigma_1.$ If the codomain of $\varphi_{\tau_2}$ is in $\widetilde{\Delta}^\star,$ then we let $\varphi_{\tau_2}\circ\chi\circ \varphi_\tau\in\widetilde{\Delta}^\star.$ Otherwise this procedure propagates through a finite number of facets of collapsed pentachora until it terminates at a facet in $\widetilde{\Delta}^\star.$ We note that at this stage, no claim was made that $|\tri^\star|$ is a manifold or PL equivalent with $|\tri|.$

\begin{Thm}
\label{thm:edgecollapse}
 Suppose that $\mathcal{T}$ is a triangulation of a $4$--manifold which admits a c-tricoloring (respectively a ts-tricoloring) and that $E$ is a monochromatic edge which is not contained in any bubble $2$--sphere. 
If  $\mathcal{T}^*$ is obtained by collapsing $E$, then $|\tri^\star|$ is PL equivalent with $|\tri|$ and $\mathcal{T}^*$ admits a c-tricoloring (respectively a ts-tricoloring). 
\end{Thm}

\begin{proof}
Consider a monochromatic edge $E$ joining two distinct vertices $v, v^\prime$ colored $R$, without loss of generality.
Each pentachoron $\sigma$ containing $E$ is the join of $E$ and a triangular face $\Delta$ with vertices colored either $BBG$ or $BGG$.
The two tetrahedral facets $\tau, \tau^\prime$ of $\sigma$ with vertices those of $\Delta$ and one of $v, v^\prime$ respectively are identified when $E$ is collapsed.
Clearly the collapse of $E$ preserves the tricoloring and each monochromatic subgraph is either unchanged or in the case of $\Gamma_R$ has the edge $E$ collapsed to a vertex. Whence the property of c-tricoloring is preserved.
Moreover the result of collapsing is a new manifold PL homeomorphic to the original one if the collapsing map is cell-like, i.e the inverse image of a point in the identification space after collapsing is either a point or a finite tree. This also implies that the property of ts-tricoloring is preserved. We claim this map is cell-like when $E$ is not contained in any bubble $2$--spheres. 

The boundary of the collection of pentachora containing $E$ is the suspension of the link $S$ of $E$, where $S$ is the set of triangular faces $\Delta$ as in the previous paragraph. So this boundary is the union of two cones over $S$, with cone points $v,v^\prime$. Note that $S$ is obtained from a 2-sphere with vertices all colored $B,G$ and any identifications of these cones must preserve the colorings. 

Suppose a sequence $s_1,s_2, \dots s_k$ of pairs of edges joining two vertices of $S$ either both to $v$ or both to $v^\prime$, are identified. Assume also that for $s_i, s_{i+1}$ and $s_k,s_1$, one of the edges of each pair is collapsed to the same image. We can label the pairs of edges by $s_1=\{E_1^+,E_2^- \}, \dots  s_k=\{E_{2k}^+,E_1^- \}$. This produces a bubble 2-sphere and a loop in the inverse image of the points in these edges after collapsing. In fact, there is a singular foliation of the bubble $2$--sphere by loops which are inverse images of points. The effect of collapsing is to map the bubble $2$--sphere to an interval, which induces a surgery on the manifold.   This shows why the absence of  bubble $2$--spheres is necessary and sufficient to ensure that collapsing $E$ corresponds to a cell-like map.   

Finally see \cite{Lacher} for a general discussion of cell-like mappings in dimensions other than 4. We can construct the PL homeomorphism directly, since the cell-like collapsing map is of a simple type. We follow a method of J.W.~Cannon, factoring the edge collapse into a sequence of small collapses, using the skeleta of the triangulation. 

Consider, for example, the inverse image of an edge $E^\prime$, where the inverse image of each interior point is a finite tree and the inverse image of one vertex is a single vertex and of the other vertex of $E^\prime$ is the edge $E$. The resulting $2$--complex $F$ is the union of a cone over the tree and the mapping cylinder of the map of the tree to the edge $E$, as can be seen by decomposing over the inverse image of the midpoint of $E^\prime$. Notice that $F$ can be viewed as a tree of $2$-simplices.
$F$ can be collapsed by homotoping leaf $2$--simplices onto two of their boundary edges, one of which is $E$ and the other is where the $2$--simplex connects onto the rest of $F$, one $2$--simplex at a time. 

After collapsing all such inverse images of edges $E^\prime$, we can go on to collapse inverse images of $2$--simplices etc. Note each small collapse is then of a PL embedded ball, namely a simplicial collapse of an embedded simplex. But it is elementary to verify that collapsing such a ball gives a map which can be approximated by a PL homeomorphism. So this completes the proof. 
\end{proof}

Suppose $\mathcal{T}$ is a triangulation of a $4$--manifold which admits a c-tricoloring (respectively a ts-tricoloring). If, after each edge collapse, there are no bubble $2$--spheres, then we can perform a series of edge collapses to obtain a new triangulation $\mathcal{T}^*$ which has a c-tricoloring (respectively a ts-tricoloring) where each monochromatic graph has a single vertex, so is a wedge of circles.
In particular, $\mathcal{T}^*$ has precisely three vertices, one of each color.
It is an interesting problem to determine whether such a series of triangulations without bubble $2$--spheres can be found.

\subsection{Simplifying moves}

\begin{Thm}
Suppose $\mathcal{T}$ is a triangulation of a $4$--manifold which admits a tricoloring (respectively a c-tricoloring).
\begin{itemize}
\setlength\itemsep{0em}
\item Suppose that $\mathcal{T}^*$ is obtained by performing a 5--1 Pachner move on $\mathcal{T}$.
	Then $\mathcal{T}^*$ admits a tricoloring (respectively a c-tricoloring).
\item Suppose that $\mathcal{T}^*$ is obtained by performing a 4--2 Pachner move on $\mathcal{T}$. Then $\mathcal{T}^*$ admits a tricoloring.
\item Assume $\mathcal{T}^*$ is obtained by performing 3--3 Pachner move on $\mathcal{T}$ and the three pentachora involved in the 3--3 move have six vertices with three sets of pairs of colors.
	Then $\mathcal{T}^*$ admits a tricoloring.
\end{itemize}
\end{Thm}

\begin{proof}
A Pachner move on a $4$--dimensional triangulation $\mathcal{T}$ consists of replacing a set of pentachora which embed in a $5$--simplex $\Omega$ by the complementary set of pentachora in $\Omega$.
Now $\Omega$ has six vertices and pentachora facets.

The facets coming from $\mathcal{T}$ are tricolored.
So we see immediately in the cases of 5--1 and 4--2 Pachner moves, that the vertices of $\Omega$ must be colored $RRBBGG$.
For if there were three vertices of the same color, then the number of tricolored facets would be at most $3$.
But if there are two vertices of each color, then every facet is tricolored.
Hence in this case, every Pachner move associated with $\Omega$ yields a tricolored triangulation.

It remains to check if $\mathcal{T}$ is c-tricolored, so is $\mathcal{T}^*$.
In the case of a 5--1 Pachner move, we can visualise this as replacing the cone on the facets of a pentachoron by the pentachoron.
Suppose the facets have vertices which are colored $RRBBG$, without loss of generality.
Then the central cone point is colored $G$.
If the monochromatic subgraphs are connected before we do this Pachner move, it is easy to see they are still connected afterwards.
For only $\Gamma_G$ changes by deleting the edge joining the two vertices colored $G$.
This edge clearly has a leaf vertex of the tree $\Gamma_G$ at the central cone point colored $G$.
So this does not disconnect the subgraph.

Consider next a 4--2 Pachner move.
Here, four pentachora share a common edge $E$.
Each pentachoron can be viewed as the join of $E$ and a triangular face $\sigma_i, 1 \le i \le 4$.
The four triangles are faces of a tetrahedral facet $\Pi$ and the replacement can be viewed as two pentachora sharing a facet $\Pi$.

Suppose first that the two vertices of $E$ have different colors, say $BG$ without loss of generality.
Then $\Pi$ has vertices with colors $RRBG$.
As in the case of the 5--1 Pachner move, the monochromatic edges deleted under our 4--2 move end at leaf vertices at $E$ so do not disconnect the monochromatic graphs.

Finally assume that the vertices of $E$ both have colors say $G$.
In this case, $\Pi$ has vertices with colors $RRBB$.
In particular, the monochromatic subgraphs $\Gamma_R, \Gamma_B$ do not change, whereas $\Gamma_G$ has $E$ deleted.
So this may disconnect $\Gamma_G$ and we may change a c-tricoloring into a tricoloring.

The case of a 3--3 Pachner move where the three pentachora have six vertices colored $RRBBGG$ is similar to the case of a 4--2 move.
In fact such a move only gives a c-tricoloring from an initial c-tricoloring, if the three pentachora share a triangular face with vertices colored $RBG$.
\end{proof}

\begin{Thm}
Assume $\mathcal{T}$ is a triangulation of a $4$--manifold which admits a tricoloring (respectively a c-tricoloring).

Suppose that $\mathcal{T}^*$ is obtained by performing a 0--2 or a 2--0 move on $\mathcal{T}$.
Then $\mathcal{T}^*$ admits a tricoloring (respectively a c-tricoloring).
\end{Thm}

\begin{proof}
To perform a 0--2 move, we pick a facet $\Pi$ and split it into a ``pillow''.
We then introduce an interior vertex $v$ in the pillow and cone to the boundary.
Without loss of generality assume the vertices of $\Pi$ are colored $RRBG$.
Then $v$ can be labeled either $B,G$.
It is easy to see the effect on the monochromatic graphs is that only the graph with the color of $v$ is changed.
In particular, this monochromatic graph has an extra loop added, so the property of being connected or not is preserved.

A 2--0 move is the opposite of a 0--2 move, so the same argument applies to establish that all the tricoloring properties of the triangulation are preserved.
\end{proof}


\section{Examples}
\label{sec:implementation}

Once a tricoloring is found for a triangulation, one needs to check the two properties that the monochromatic graphs are connected and that the 3--dimensional trisection submanifolds have 1--dimensional spines. An example of a tricolorable triangulation, where the former property holds but the latter fails, is the triangulation of $S^1 \times S^3$ with isomorphism signature \texttt{gLAAMQacbdcdefffcaTava4acavayaWaZa2a} \cite{Burton}. This has six pentachora and three vertices. The three monochromatic graphs are circles, but the central surface consists of three pairwise disjoint 2--tori. Two ts-tricolourable triangulations with six pentachora and three vertices of this manifold were found by Jonathan Spreer. These have isomorphism signatures 
\texttt{gLMPMQccdeeeffffaaaa9aaaaaaaaaaaaa9a} and \texttt{gLwMQQcceeeffeffaaaaaaaaaaLaLaLaLaLa} respectively. The central surface in each case is a 2--torus, and all 4--dimensional and 3--dimensional handlebodies have genus one.

Since $\pi_1(S^1 \times S^3) \isom \mathbb{Z}$ and the fundamental group of the central surface surjects onto the fundamental group of each handlebody (see the proof of Proposition~5 in \cite{RT2016}), the central surface of any trisection of $S^1 \times S^3$ has genus at least one.
Hence we recover the result of \cite{GK} that:

\begin{Coro}
The trisection genus of $S^1 \times S^3$ is one.
\end{Coro}

We note that $S^1 \times S^3$ also has a triangulation with just two pentachora, \texttt{cMkabbb2aHaua2a}, so application of Corollary~\ref{cor:genus bound arbitrary} merely gives a bound of $\genus(\Sigma) \le 120$.
This highlights the fact that whilst our main bound gives a first and linear upper bound on the minimal genus of a central surface, in practice this may be far from optimal.

\begin{Que}
Are there interesting families of $4$--manifolds for which there exists an algorithm to compute a multisection of minimal genus for each member of the family?
\end{Que}

\begin{Que}
Are there interesting families of $4$--manifolds for which one can find ts-tricolorable triangulations in which the central surface is of minimal genus?\\
\end{Que}

Our current technique allow us to first determine a ts-tricolored triangulation with a large number of pentachora, and then to collapse this to a smaller triangulation.
This could be improved with better heuristics to produce tricolorable triangulations to which we apply 2--4 moves in order to obtain a ts-tricolored triangulation.
An indispensable tool for experimentation would be an implementation of our algorithms and heuristic procedures in Regina~\cite{Regina}.


\section*{Acknowledgements}
The authors thank the American Institute for Mathematics and the organizers of the workshop \emph{Trisections and low-dimensional topology}, where this work was initiated, and thank Jeff Meier and Jonathan Spreer for helpful discussions.
Hass was partially supported by NSF grant DMS-1719582-0.
The research of Rubinstein and Tillmann is partially supported under the Australian Research Council's Discovery funding scheme (project number DP160104502).
Tillmann thanks the DFG Collaborative Center SFB/TRR 109 at TU Berlin, where parts of this work have been carried out, for its hospitality.


\bibliographystyle{plain}
\bibliography{Trisection_Algorithms}


\address{Mark Bell\\Department of Mathematics, University of Illinois, Urbana, IL 61801, USA\\ 
{mcbell@illinois.edu\\-----}}

\address{Joel Hass\\Department of Mathematics, University of California Davis, Davis, CA 95616, USA\\ 
{hass@math.ucdavis.edu\\-----}}

\address{J. Hyam Rubinstein\\School of Mathematics and Statistics, The University of Melbourne, VIC, 3010, Australia\\ 
{joachim@unimelb.edu.au\\-----}}

\address{Stephan Tillmann\\School of Mathematics and Statistics, The University of Sydney, NSW 2006, Australia
\\ {stephan.tillmann@sydney.edu.au}}

\Addresses

\end{document}